\documentclass[letterpaper, 10 pt, conference]{ieeeconf}
\IEEEoverridecommandlockouts   
\usepackage{amsmath}
\usepackage{amssymb,amsfonts,graphicx}
\usepackage{mathrsfs}  

\newtheorem{Theorem}{Theorem}
\newtheorem{Assumption}{Assumption}

\newtheorem{Proposition and Definition}[Theorem]{Proposition and Definition}
\newtheorem{Lemma}{Lemma}

\newtheorem{Definition}{Definition}
\newtheorem{Example}{Example}

\newtheorem{Remark}{Remark}

\newtheorem{Notation}{Notation}

\title{Realization theory of recurrent neural networks and rational systems}
\author{Thibault Defourneau, Mih\'aly Petreczky
\thanks{Centre de Recherche en Informatique, Signal et Automatique de Lille (CRIStAL), CNRS, Ecole Centrale, Université de Lille, Avenue Carl Gauss, 59650 Villeneuve-d'Ascq, France (e-mail:thibault.defourneau@univ-lille.fr, mihaly.petreczky@centralelille.fr) }
\, \thanks{This work was partially funded by CPER Data project, which is co-financed by European Union with the financial support of European Regional Development Fund (ERDF), French State and the French Region of Hauts-de-France. This work was partially funded by Agence Nationale pour la Recherche Project ROCC-SYS under the Grant agreement ANR-14-C27-0008}
}

\begin{document}
\maketitle
\begin{abstract}
In this paper, we show that, under mild assumptions, input-output behavior of a continuous-time recurrent neural network (RNN) can be represented by a  rational or polynomial non-linear system. The assumptions concern the activation function of RNNs, and they are satisfied by many classical activation functions such as the hyperbolic tangent. We also present an algorithm for constructing the polynomial and rational system. This embedding of RNNs into rational systems  can be useful for stability, identifiability, and realization theory for RNNs, as these problems have been studied for polynomial/rational systems. In particular, we use this embedding for deriving necessary conditions for realizability of an input-output map by RNN, and for deriving sufficient conditions for minimality of an RNN.  
\end{abstract}

\section{Introduction}

One of the challenges in machine learning is to provide a mathematical theory for analyzing learning algorithms. Recently, there has been a surge of interest in the use of neural networks, leading to the emergence of the field of deep learning. One of the most widespread models used in deep learning are recurrent neural networks (RNNs). RNNs can seen as non-linear dynamical systems equipped with an internal state, input and output. Learning such an RNN from data is equivalent to estimating the parameters of the RNN, viewed as a dynamical system. That is, learning algorithms for RNNs correspond to system identification algorithms, and developing a mathematical theory for learning RNNs is equivalent to developing system identification for RNNs. There is a rich literature on system identification, in particular on system identification for linear systems \cite{LjungBook}. Note that linear dynamical systems are a particular class of RNNs. 

One of the principal building blocks of system identification theory for linear systems is realization theory. Realization theory can be viewed as an attempt to solve an idealized system identification problem, where there is infinite data, not modelling error, etc. In general, the aim of \emph{realization theory} is to understand the relationship between an observed behavior and dynamical systems producing this observed behavior. For the particular case of RNNs, the main questions of realization theory can be stated as follows:
\begin{enumerate}
\item Which class of observed behaviors (input-output maps) can be represented by an RNN ?
\item How can we characterize minimal RNNs (RNNs of the least complexity) representing a certain observed behavior ? What is the appropriate definition of minimality (smallest number of neurons, etc.) for RNNs, are minimal RNNs are related by some transformation ?
\item Is there a constructive procedure for constructing a RNN representation from input-output behavior which can be proven to be mathematically correct ?
\end{enumerate}

For linear systems, realization theory \cite{kailath,LindquistBook} has been useful for system identification, for example it helped to address identifiability, canonical forms and gave rise to subspace identification algorithms.
We expect that realization theory of RNNs will lead to similarly useful results for the latter dynamical systems.

In order to develop realization theory of RNNs, we embed RNNs into the class  of rational systems, and then we use realization theory of rational/polynomial systems in order to derive new results on realization theory of RNNs. Then the proposed embedding could also be useful beyond realization theory, as it could open up the possibility of studying for example stability of RNNs by using existing results on stability of rational/polynomial systems. For example, for a large class of activation functions, RNNs can be embedded into a homogeneous polynomial systems and the states of the latter systems are continuous functions of the states of the original RNNs. This then opens up the possibility of studying stability of RNNs by using stability theory of homogeneous systems \cite{HomStability1,HomStability2}.


In this paper, we consider RNNs in continuous-time, in order to avoid some technical difficulties encountered in the discrete-time case. Note that discrete-time RNNs can be viewed as discretizations of continuous-time RNNs, i.e. they arise by discretizing the differential equations describing the state evolution of continuous-time RNNs, which indicates that the results of this paper could be relevant for the discrete-time case.

In this paper we assume that the high-order derivatives of the activation function satisfy a polynomial equation. Several widely-used activation functions have this property. 
\begin{itemize}
\item We show that an input-output map can be realized by a RNN, only if it can be realized by a rational system, i.e. a non-linear system defined by vector fields and readout maps which are fractions of polynomials. We present an explicit construction of such a rational system.
\item We present a necessary condition for existence of a realization by RNNs, using results from realization theory of rational systems. This necessary condition is a generalization of the well-known rank condition for Hankel matrices of linear systems.
\item We formulate sufficient conditions for observability/reachability/minimality of RNNs, using existing realization theory for rational systems \cite{JanaSIAM,Jana,Bartoszewicz,SontagWang}.
\end{itemize}

Note that RNNs could be viewed as analytic systems and one could try to apply realization theory of analytic systems \cite{J,Isidori,NLCO}. 
However, realization theory for analytic systems is not computationally effective, i.e., there are no algorithms for checking minimality, deciding equivalence of two systems, transforming a system to a minimal one, etc. This is inherent to the system class: analytical functions do not have a finite representation. This is in contrast to rational and polynomial systems, where tools from computer algebra could be used
\cite{JanaCDC2016}. In addition, since rational and polynomial systems have much more algebraic structures than analytic systems, we expect them to yield richer results for realization theory of RNNs than analytical systems. In fact, the conditions for  observability/reachability/minimality of RNNs which are derived in this paper
are less restrictive than those which can be obtained by viewing RNNs as analytic systems \cite{J,Isidori,NLCO}.

To the best of our knowledge, the results of the paper are new. RNNs have been widely used in the machine learning literature, both in discrete-time and continuous-time, \cite{MachineLearningBook1,MachineLearningBook2}. Observability of RNNs was studied in \cite{ASo}, controllability in \cite{QSc} and minimality in \cite{ASm}. This paper was inspired by \cite{ASo,QSc,ASm}, but in contrast to \cite{ASo,QSc,ASm}, we do not use any assumption on the structure of the weight matrices, except for observability issues. This means that the results of this paper can be applied even when the results of \cite{ASo,QSc,ASm} are not applicable. 

In Section \ref{def} we present the basic notation and terminology, and we present the formal definition of RNNs, rational systems and their input-output maps. In Section \ref{sect:main:embbed} we present the construction of the rational system which realizes the same input-output map as an RNN. In Section \ref{sect:exist}, we use the results of Section \ref{sect:main:embbed} to derive necessary conditions for existence of a realization by RNN. Finally, in Section \ref{sect:min} we use the results of Section \ref{sect:main:embbed} to derive sufficient conditions for minimality of RNNs. An extended version of this paper containing detailed proofs can be found in the technical report \cite{TMArxive}.

\section{Basic definitions} \label{def}

In this section, we give fix some notation and we recall some algebraic tools necessary for this paper. Then we recall the definition of RNNs and of rational systems.

\subsection{Preliminaries}

We use the standard terminology and notation from commutative algebra and algebraic geometry see \cite{kunz:1985,zariski:samuel:1958,cox:little:oshea:1992}.
In particular, by $\mathbb{R}[X_1, \dots ,X_n]$ we denote the algebra of real polynomials in $n$ variables and
by $R(X_1,\ldots,X_n)$ we denote the quotient field of $\mathbb{R}[X_1, \dots ,X_n]$. We refer to the element of  $R(X_1,\ldots,X_n)$ as rational functions in 
$n$ variables. 
If $S$ is an integral domain over $\mathbb{R}$ then the transcendence degree $\mathrm{trdeg} S$ of $S$ over $\mathbb{R}$ is defined as the transcendence degree over $\mathbb{R}$ of the field $F$ of fractions of $S$ and it equals the greatest number of algebraically independent elements of $F$ over $\mathbb{R}$.

\subsection{Recurrent neural networks.} 

Below we define formally what we mean by recurrent neural networks in continuous-time. We will follow the notation of \cite{ASo,QSc,ASm}. 

\begin{Definition} \label{neural_network}
A \emph{recurrent neural network}, abbreviated by RNN, with input-space 
$\mathcal{U} \subset \mathbb{R}^m$ and output-space $\mathbb{R}^{\mathrm{p}}$,  is a dynamical system 
\begin{equation} \label{RNN_equation}
\Sigma: \left\lbrace \begin{array}{lll} 
\dot{x}(t) = \overrightarrow{\sigma} \big( A x(t) + B u(t) \big) \\
x(0) = x_0  \\
y(t) = C x(t) \end{array} \right. 
\end{equation}
where  
\begin{itemize}
\item $\sigma: \mathbb{R} \rightarrow \mathbb{R}$ is a continuous \emph{globally} Lipschitz scalar function. It is called the \emph{activation function} (of the RNN).
\item $A \in \mathbb{R}^{n \times n}$, $B \in \mathbb{R}^{n \times m}$ and $C \in \mathbb{R}^{\mathrm{p} \times n}$ are matrices, called \emph{weight matrices},
\item the map $\overrightarrow{\sigma}: \mathbb{R}^n \rightarrow \mathbb{R}^n$ is defined by 
\begin{equation*}
\overrightarrow{\sigma}: (x_1, \ldots, x_n)^T \mapsto (\sigma(x_1), \ldots, \sigma(x_n))^T \, ,
\end{equation*}
\item $u(t)$ is an input, $x(t) \in \mathbb{R}^n$ is the state and $y(t)$ is the output at time $t$, and $x_0 \in \mathbb{R}^n$ is the initial-state.
\end{itemize}
We denote such a system by $\Sigma = (A, B, C, \mathcal{U}, \sigma, x_0)$.
\end{Definition}

If we take $\sigma$ the identity map in Definition \ref{neural_network}, it is the same as a linear system in control theory. It follows that RNNs provide a class of semi-linear systems, for which one might expect that a theory closer to that of linear system than in case of general non-linear smooth systems. 

Next we define formally what we mean by a solution and an input-output map of a system $\Sigma = (A, B, C, \mathcal{U},\sigma, x_0)$. To this end, in the sequel, we denote by $PC([0; + \infty[, X)$ the set of piecewise-continuous functions from $[0; + \infty[$ to $X$, where $X \subseteq \mathbb{R}^k$, $k > 0$. 

\begin{Definition} \label{solution_of_a_RNN}
A triple $(x, u, y)$ is a \emph{solution} of an RNN $\Sigma = (A, B, C, \mathcal{U},\sigma, x_0)$ if $u \in PC([0; + \infty[, \mathcal{U})$, $x: [0; + \infty [ \rightarrow \mathbb{R}^n$, $y: [0; + \infty [ \rightarrow \mathbb{R}^{\mathrm{p}}$, $x$ is absolutely continuous, and \eqref{RNN_equation} holds.
\end{Definition}

\begin{Remark} \label{remark_solution_RNN}
Let $\Sigma = (A, B, C, \mathcal{U},\sigma, x_0)$ be an RNN. As the activation function $\sigma$ is globally Lipschitz, we know that, for every piecewise continuous input $u: [0; + \infty [ \rightarrow \mathcal{U}$, there exists a unique absolutely continuous functions $x: [0; + \infty [ \rightarrow \mathbb{R}^n$ and a function $y: [0; + \infty [ \rightarrow \mathbb{R}^{\mathrm{p}}$ such that $(x, u, y)$ is a solution of $\Sigma$.
\end{Remark}

In this paper, we focus on solutions $(x, u, y)$ of an RNN such that $u$ is \emph{piecewise constant}.
 
\begin{Notation}[Piecewise-constant inputs]
We denote by $\mathcal{U}_{pc}$ the set of all piecewise-constant functions of the form 
$u:[0; + \infty [ \rightarrow \mathcal{U}$.
\end{Notation}

\begin{Definition} \label{realization_response_map1}
Let $p: \mathcal{U}_{pc} \rightarrow PC([0; + \infty[, \mathbb{R}^{\mathrm{p}})$ be an input-output map, and let $\Sigma = (A, B, C, \mathcal{U},\sigma, x_0)$ be an RNN.
$\Sigma$ is said to be a \emph{realization} of the input-output map $p$ if for every $u \in \mathcal{U}_{pc}$, the unique solution $(x,u,y)$, $x(0)=x_0$ of $\Sigma$
is such that $p(u)=y$.
\end{Definition}

\subsection{Rational and polynomial systems.} 

 Informally, rational respectively polynomial systems are control systems in continuous time, whose differential equations and readout maps are rational functions, i.e. they are fractions of polynomials, respectively polynomial functions.
\begin{Definition}[Polynomial and rational systems] \label{rational_system}
A \emph{rational system} with input-space $\mathcal{U} \subset \mathbb{R}^m$, state-space $\mathbb{R}^n$ and output-space $\mathbb{R}^{\mathrm{p}}$ is a dynamical system as 
\begin{equation} \label{diff_equation_rational_system}
\mathscr{R}: \left\lbrace \begin{array}{lll}
\dot{\upsilon_i}(t)=\frac{P_{i,u(t)}(\upsilon(t))}{Q_{i,u(t)}(\upsilon(t))} , i=1,\ldots,n \, , \; \upsilon(0)=\upsilon_0 \\
y_k(t) =\frac{h_{k,1}(\upsilon(t))}{h_{k,2}(\upsilon(t)))} \, , k=1,\ldots,\mathrm{p}
\end{array} \right.
\end{equation}
where 
\begin{itemize}
\item $u(t)$ is the  input, $\upsilon(t)=(\upsilon_1(t),\ldots,\upsilon_n(t))^T$ is the state and $y(t)=(y_1(t),\ldots,y_k(t))^T$ is the output at time $t$. Moreover $\upsilon_0 \in \mathbb{R}^n$ is the initial state;
\item $h_{k,1},h_{k,2}$, $k=1,\ldots,\mathrm{p}$ are non-zero polynomials in $n$ variables
     and for all $\alpha \in \mathcal{U}$, $i=1,\ldots,n$,
      $P_{i,\alpha},Q_{i,\alpha}$ are polynomials in $n$ variables, $Q_{i,\alpha}$ is non-zero.  
\end{itemize} 
We will identify the rational system $\mathscr{R}$ with the tuple
$(\{P_{i,\alpha},Q_{i,\alpha}\}_{i=1,\ldots,n, \alpha \in \mathcal{U}}, \{h_{k,1},h_{k,2}\}_{k=1}^{\mathrm{p}},\mathcal{U}, \upsilon_0)$.
We will say that $\mathscr{R}$ is \emph{polynomial}, if $h_{k,2}=1$, $Q_{i,\alpha}=1$ for all $k=1,\ldots,\mathrm{p}$, $i=1,\ldots,n$, $\alpha \in \mathcal{U}$. 
\end{Definition}
Informally, a rational system is a non-linear control system for which the right-hand sides of the differential
equation and output equations are rational functions, i.e. fractions of two polynomials. Note that the existence and uniqueness of a state trajectory of a rational system requires some care, as in our definition we did not exclude the possibility that $x(t)$ passes through the zero set of a denominator $Q_{i,\alpha}$. 
In order to avoid technical difficulties, we define a solution of a rational system  as follows:

\begin{Definition} \label{solution_of_rational_systems}
A triplet $(\upsilon, u, y)$ is a solution of a rational system $\mathscr{R}$ of the form \eqref{diff_equation_rational_system}, 
if the input $u: [0; + \infty[ \rightarrow \mathcal{U}$ is piecewise constant, the state $\upsilon: [0; + \infty [ \rightarrow \mathbb{R}^n$ is absolutely continuous,  the output $y: [0; + \infty [ \rightarrow \mathbb{R}^{\mathrm{p}}$ is piecewise continuous, and they satisfy
\begin{equation} \label{rational_equation}
\begin{array}{lll}
\dot{\upsilon_i}(t) \; Q_{i,u(t)}(\upsilon(t)) = P_{i,u(t)}(\upsilon(t)) \, \\
y_k(t) \; h_{k,2}(\upsilon(t)) = h_{k,1}(\upsilon(t)) \, ,
\end{array}
\end{equation}
for $1 \leqslant i \leqslant n$, and $1 \leqslant k \leqslant \mathrm{p}$.
\end{Definition}
\begin{Remark}[Uniqueness of a solution]
Let $(\upsilon, u, y)$ be a solution of a rational system as above. If, for any $t \geqslant 0$ such that $Q_{i,u(t)}(\upsilon(t)) \neq 0$ and $h_{k,1}(\upsilon(t)) \neq 0$, for all $i \in \{1, \ldots, n\}$ and $k \in \{1, \ldots, \mathrm{p}\}$, then \eqref{diff_equation_rational_system} holds. Hence, by uniqueness of a solution of an analytic  differential equation,
for any initial state $\upsilon_0 \in \mathbb{R}^n$ such that $Q_{i,u(t)}(\upsilon_0) \neq 0$ and $h_{k,1}(\upsilon_0) \neq 0$, there exist at most one solution $(\upsilon,u,y)$ of $\mathscr{R}$ such that $\upsilon(0)=\upsilon_0$. 
In particular, if $\mathscr{R}$ is  a polynomial system, then for any initial state $\upsilon_0 \in \mathbb{R}^n$ there exist at most one solution $(\upsilon,u,y)$ of $\mathscr{R}$ such that $\upsilon(0)=\upsilon_0$. 
\end{Remark}
\begin{Definition} \label{realization_response_map}
Let $p: \mathcal{U}_{pc} \rightarrow PC([0; + \infty[, \mathbb{R}^{\mathrm{p}})$ be an input-output map, and let 
$\mathscr{R}$ 
be a rational system of the form \eqref{diff_equation_rational_system}. We say that $\mathscr{R}$ realizes $p$, if for any $u \in \mathcal{U}_{pc}$ there exists  
a solution $(\upsilon, u, y)$ of $\mathscr{R}$ such that $\upsilon(0)=\upsilon_0$ and $p(u)=y$.
\end{Definition}


I


\section{Embedding of a class of recurrent neural network realizations into rational realizations.}  \label{sect:main:embbed}
 In this section  we show that an RNN realization of a given input-output map imply the existence of a rational system which is a realization of the same map. Moreover, we present the construction of such a rational realization. 


In order to state the announced result, we have to restrict the class of activation function $\sigma$
by introducing the following assumption.
\begin{Assumption}[$A1$]
\label{assum1}
 The function $\sigma: \mathbb{R} \rightarrow \mathbb{R}$ is analytic, and there exist an integer $N > 0$ and $N$ analytic functions $\xi_1, \ldots, \xi_N: \mathbb{R} \rightarrow \mathbb{R}$ such that
\begin{equation}\label{Assumption1}
\hspace*{-0.60cm}{
\left\lbrace
\begin{array}{ll}
 \sigma \; V_0(\xi_1, \ldots, \xi_N) = U_0(\xi_1, \ldots, \xi_N)\\[4mm]
\dot{\xi_i} \; V_i(\xi_1, \ldots, \xi_N) = U_i(\xi_1, \ldots, \xi_N) \, , \; \text{if} \; 1 \leqslant i \leqslant N
\end{array} \right.}
\end{equation}
where $U_k, V_k$ are polynomials in $N$ variables. 
\end{Assumption}

Assumption $(A1)$ involves existence of analytic functions $\{\xi_i\}_{i=1}^N$ and hence it is not easy to check it. In fact, Assumption $(A1)$ can be replaced by the following hypothesis, which involves only derivatives of the activation function $\sigma$.

\begin{Assumption}[$A2$]
\label{assum2}
The function $\sigma: \mathbb{R} \rightarrow \mathbb{R}$ is analytic, and there exist an integer $N > 0$ 
and a no-zero polynomial $Q$ in $N+1$ variables, such that 
\begin{equation} \label{Assumption2}
Q(\sigma, \sigma^{(1)} \ldots, \sigma^{(N)}) = 0\\[2mm]
\end{equation}
where $\sigma^{(i)}$ denotes the $i$-th derivative of $\sigma$. 
\end{Assumption}

\begin{Lemma}[Equivalence of $(A1)$ and $(A2)$]
\label{lemma:assum:eq}
 A function $\sigma$ satisfies Assumption $(A1)$ if and only if it satisfies Assumption $(A2)$. 
\end{Lemma}
The proof of Lemma \ref{lemma:assum:eq} is presented in \cite{TMArxive}.

Now we show that some widely used activation functions satisfy Assumption $(A2)$.

\begin{Example} \label{example_activation_functions}
Let consider an analytic RNN of a response map for which the activation function $\sigma$ is the hyperbolic tangent $th$, or the sigmoid function $S$ given below:
\begin{equation*}
\forall x \in \mathbb{R} \, , \hspace*{2mm} th(x) = \frac{e^x - e^{-x}}{e^x + e^{-x}} \, , \hspace*{2mm} S(x) = \frac{1}{1 + e^{-x}} \, .
\end{equation*}
These functions are \emph{analytic} and satisfy a differential polynomial equation, more precisely the hyperbolic tangent verifies $y^{(1)} = 1 - y^2$ with $y(0)=0$, and the sigmoid function satisfies $y^{(1)} = y (1 -y)$ with $y(0) = \frac{1}{2}$. It follows that Assumption $(A2)$ holds.
\end{Example}

Then we restrict the set of input maps, by supposing the following assumption.

\begin{Assumption}[Finite input set]
\label{assum2}
In the rest of the paper we assume that $\mathcal{U} \subset \mathbb{R}^m$ is a \textbf{finite} set.
\end{Assumption}

\begin{Notation}
We denote by $| \mathcal{U} |$ the cardinality of $\mathcal{U}$. We set $\mathcal{U} = \{\alpha_1, \ldots, \alpha_K\}$, where $\alpha_i \in \mathbb{R}^m$, and $\alpha_i \neq \alpha_j$ if $i \neq j$. In that case, we have $| \mathcal{U} | = K$.
\end{Notation} 

The assumption that $\mathcal{U}$ is finite is not an overly restrictive one, and it is satisfied in many
potential applications. 

Next we present the main theorem of this paper.

\begin{Theorem}[Embedding RNNs into rational systems] 
\label{Theoremi:main}
Let $\sigma: \mathbb{R} \rightarrow \mathbb{R}$ be a globally Lipschitz function which satisfies $(A1)$
and assume that $\mathcal{U}$ is finite. Consider an input-output map $p: \mathcal{U}_{pc} \rightarrow PC([0; + \infty[, \mathbb{R}^p)$.
 If $\Sigma$ is a RNN with activation function $\sigma$ and input space $\mathcal{U}$, and 
 $\Sigma$ is a realization of $p$, then there exists a rational system which is a realization of $p$.
\end{Theorem}
The proof of Theorem \ref{Theoremi:main} is presented in \cite{TMArxive}.

The proof of the theorem relies on defining a rational system associated with the RNN.
In order to define this rational system without excessive notation, in the sequel we identify the sum of fraction of multi-variable polynomials $\sum_{k=1}^{N} \frac{P_i}{Q_i}$ with the fraction obtained by bringing all summands to
the same denominator, i.e., $\frac{\sum_{k=1}^{N} P_i\Pi_{r=1, r \ne k}^{N}Q_i}{\Pi_{k=1}^{N} Q_i}$. 

\begin{Definition} \label{rational_system_R(Sigma)}
Let $\Sigma$ be an RNN, whose activation function satisfies $(A1)$, assume that $\xi=(\xi_1,\ldots, \xi_N)^T$ satisfies \eqref{Assumption1}. Define the \emph{rational system $\mathscr{R}(\Sigma)$ associated with the RNN $\Sigma=(A,B,C,\mathcal{U},\sigma,x_0)$},
$A=(a_{i,j})_{i,j=1}^{n}, C=(c_{k,i})_{k=1,\ldots,p,i=1,\ldots,n}$  as follows:
\begin{align*}
& \forall i=1,\ldots,N, ~ j=1,\ldots,n, ~ \alpha \in U: \\
& \dot \upsilon_{i,j,\alpha}(t)= \frac{U_{i}(\upsilon_{j,\alpha}(t))}{V_i(\upsilon_{j,\alpha}(t))}\{\sum_{l=1}^{n} a_{j,l}  \frac{U_0(\upsilon_{l,\beta}(t))}{V_0(\upsilon_{l,\beta}(t))}\} ~ \mbox{ if } u(t)=\beta\\
&  \upsilon_{j,\alpha}(t)=(\upsilon_{1,j,\alpha}(t),\ldots,\upsilon_{N,j,\alpha}(t)),\\
& \upsilon_{j,\alpha}(0)=\xi \big( e_j^T(Ax_0 + B\alpha) ) \\
& \dot x_j(t)=\frac{U_0(\upsilon_{j,\beta}(t))}{V_0(\upsilon_{j,\beta})} ~ \mbox{if } u(t)=\beta, \mbox{ and } x_j(0)=e_j^Tx_0, \\
& y_k(t)=\sum_{i=1}^n c_{k,i}x_i(t), ~ k=1,\ldots, \mathrm{p} \\
\end{align*}
\end{Definition}

%
\begin{Remark}[Constructing $\mathscr{R}(\Sigma)$]
\label{rem:rat:construct}
That is $\mathscr{R}(\Sigma)=(\{P_{i,\alpha},Q_{i,\alpha}\}_{i=1,\ldots,L, \alpha \in \mathcal{U}}, \{h_{k,1},h_{k,2}\}_{k=1}^{\mathrm{p}},\mathcal{U},\upsilon_0)$, where $L=n+nN|\mathcal{U}|$, and $P_{i,\alpha},Q_{i,\alpha},h_{k,1}, h_{k,2}$ are polynomials in the variables $X_1,\ldots, X_L$ of the following form: for any $i=1,\ldots,N,j=1,\ldots,n, \alpha \in \mathcal{U}$, let $\phi(i,j,\alpha) = N \cdot |\mathcal{U}| \cdot (j-1)+N\cdot (r-1)+i$ if $\alpha=\alpha_r$, $r=1, \ldots, K$, and let $X_{j,\alpha}$ denote the tuple $X_{\phi(1,j,\alpha)},\ldots, X_{\phi(N,j,\alpha})$. Note that any $k \in \{1,\ldots,L-n\}$ arises as $\phi(i,j,\alpha)$ for suitable $i=1,\ldots,N,j=1,\ldots,n, \alpha \in \mathcal{U}$. In particular, if $R$ is a polynomial in $N$ variables, then $R(X_{j,\alpha})=R(X_{\phi(1,j,\alpha)},\ldots, X_{\phi(N,j,\alpha})$ is a polynomial in $X_{\phi(1,j,\alpha)},\ldots, X_{\phi(N,j,\alpha})$. With this notation and using that $\xi=(\xi_1,\ldots,\xi_n)$ is from \eqref{Assumption1}, 
\begin{align*}
& \forall k=1,\ldots,\mathrm{p}:  \\
& h_{k,1}=1, \hspace*{2mm} h_{k,2}=\sum_{i=1}^{n} c_k X_{j+|U|Nn} \\
& \forall \; i=1,\ldots,N, ~ j=1,\ldots,n, ~\alpha \in \mathcal{U}: \\
& P_{N|\mathcal{U}|n+j,\alpha}=U_0(X_{j,\alpha}), ~ Q_{N|\mathcal{U}|n+j,\alpha}=V_0(X_{j,\alpha}) \\
& P_{\phi(i,j,\alpha),\beta}= U_i(X_{j,\alpha}) 
  (\sum_{k=1}^{n} a_{j,k} U_0(X_{k,\beta})\Pi_{r=1,r \ne k}^N 
     V_0(X_{r,\beta})) \\
 & Q_{\phi(i,j,\alpha),\beta}= V_i(X_{j,\alpha})\Pi_{r=1}^N V_0(X_{r,\beta}) \\
 & (\upsilon_0)_{\phi(i,j,\alpha)}=\xi_i \big( e_j^T(Ax_0 + B\alpha) \big), ~ (\upsilon_0)_{N|\mathcal{U}|+j}=e_j^Tx_0.
\end{align*}
Here $(\upsilon_0)_k$ denotes the $k$-th entry of $\upsilon_0 \in \mathbb{R}^L$. It is then clear that $\mathscr{R}(\Sigma)$ can be computed from the matrices $A,B,C$ and from the polynomials $\{U_i,V_i\}_{i=0}^{N}$ of Assumption $(A1)$. 
\end{Remark}

\begin{Remark}[Polynomial and homogeneous $\mathscr{R}(\Sigma)$]
If Assumption $(A1)$ is satisfied with polynomial equations, i.e. $V_i=1$, $i=0,\ldots,N$, like in examples from Example \ref{example_activation_functions}, then $\mathscr{R}(\Sigma)$ is a polynomial system. If Assumption $(A1)$ is satisfied with homogeneous polynomial equations, i.e. $V_i=1,i=0,\ldots,N$ and $U_i$ are homogeneous polynomials, then $\mathscr{R}(\Sigma)$ is a polynomial system and $h_{1,k}$, $P_{i,\alpha}$, $\alpha \in \mathcal{U}$, $i=1,\ldots,(N | \mathcal{U} | +1)n$, $k=1,\ldots,\mathrm{p}$ are homogeneous polynomials too. 
\end{Remark}

The proof of Theorem \ref{Theoremi:main} relies on the following simple result, which is interesting on its own right.  Let $\xi=(\xi_1,\ldots,\xi_n)$ is from \eqref{Assumption1}, and define $F:\mathbb{R}^n \rightarrow \mathbb{R}^L$, as $F(x)=(z_1,\ldots,z_{nN|\mathcal{U}|},x^T)^T$, and using the notation of Remark \ref{rem:rat:construct}, let $z_{\phi(i,j,\alpha)}=\xi_i(e_j^T(Ax+B\alpha))$.
\begin{Lemma} 
\label{Theoremi:main:lemma}
If $(x,u,y)$ is a solution of the RNN $\Sigma=(A,B,C,\mathcal{U},\sigma,x_0)$, then $(\upsilon,u,y)$, with $\upsilon(t)=F(x(t))$ for $t \ge 0$, is a solution of $\mathscr{R}(\Sigma)$.  
\end{Lemma}
The proof of Lemma \ref{Theoremi:main:lemma} is presented in the technical report \cite{TMArxive}.

\section{Application of the embedding theorem: existence of an RNN realization}
\label{sect:exist}

Theorem \ref{Theoremi:main} allows us to formulate a necessary condition for realizability of an input-output map by an RNN, using conditions of \cite[Theorem 5.16]{JanaSIAM} for existence of a realization by a rational system. In order to present this condition, we need to introduce additional notation and terminology. In particular, we have to define the class of input-output maps which could potentially be realized by an RNN.  

The most basic requirement for an input-output map to be realizable by a control system is causality:

\begin{Definition} \label{definition_causal}
An input-output map $p: \mathcal{U}_{pc} \rightarrow PC([0; + \infty[, \mathbb{R}^{\mathrm{p}})$ is \emph{causal} if,
for all $t \geqslant 0$ and for all $u, v \in \mathcal{U}_{pc}$  such that  $\forall s \in [0,t]: u(s) = v(s)$, it holds that $\forall s \in [0,t]: p(u)(s) = p(v)(s)$.
\end{Definition}
In other words, $p$ is causal, if $p(u)(t)$ depends only on the values of $u$ on the interval $[0,t]$. If $p$ is the input-output map of a control system, then causality must necessarily hold. 

Another  basic requirement is analyticity: if $p$ is the input-output map of a control system defined by a differential equation with analytic right-hand side, then for any piecewise-constant input, $p(u)$ should be analytic in a suitable defined sense, i.e., in the dwell time of the constant pieces of $u$.
\begin{Definition} \label{definition_analycity}
An input-output map $p: \mathcal{U}_{pc} \rightarrow PC([0; + \infty[, \mathbb{R}^{\mathrm{p}})$ is \emph{analytic} if, for all $k \in \{1, \ldots, \mathrm{p}\}$ and for all $\alpha_1,\ldots,\alpha_l \in \mathcal{U}$, $l > 0$, the function $\phi_{p,k,\alpha_1,\ldots,\alpha_l}: ([0,+\infty[)^l \rightarrow \mathbb{R}$ is analytic, where 
\begin{equation}
\label{definition_analycity:eq}
\begin{split}
& \phi_{p,k,\alpha_1,\ldots,\alpha_l}(t_1, \ldots, t_l)=p_k(u_{t_1,\ldots,t_l}^{\alpha_1,\ldots,\alpha_l})(T_l) \\
& u_{t_1,\ldots,t_l}^{\alpha_1,\ldots,\alpha_l}(t)=\left\{\begin{array}{rl}
 \alpha_i & \mbox{ if } \; t \in [T_{i-1}, T_i[, \hspace*{1mm}  i=1,\ldots,l \\
 \alpha_l & \mbox{ if } \; t \ge T_l
 \end{array}\right.    \\
& T_0=0 \, , \hspace*{1mm} T_i=\sum_{j=1}^{i} t_j, \hspace*{1mm} i=1,2,\ldots,l.
\end{split}
\end{equation}
\end{Definition}
\begin{Notation}
We denote by 
$\mathcal{A}(\mathcal{U}_{pc})$
the set of causal analytic input-output maps of the form $p:\mathcal{U}_{pc} \rightarrow PC([0; + \infty[, \mathbb{R})$.
\end{Notation}
Note that the set $\mathcal{A}(\mathcal{U}_{pc})$ forms an algebra over the field of real numbers 
with the usual point-wise addition, multiplication and multiplication by scalar. 

\begin{Remark} \label{algebra_io-map_integral_domain}
It can be shown that the algebra $\mathcal{A} (\mathcal{U}_{pc})$ is isomorphic to the ring of functions  $\mathcal{A} \big( \widetilde{\mathcal{U}_{pc}} \rightarrow \mathbb{R} \big)$ defined in \cite[Definition 4.3]{JanaSIAM}, if we take the set of all piecewise-constant input functions defined on a finite interval as the set of admissible inputs $\widetilde{\mathcal{U}_{pc}}$. Therefore, $\mathcal{A}(\mathcal{U}_{pc})$ has the same algebraic properties as $\mathcal{A} \big( \widetilde{\mathcal{U}_{pc}} \rightarrow \mathbb{R} \big)$, in particular, by \cite[Theorem 4.4]{JanaSIAM}, $\mathcal{A}(\mathcal{U}_{pc})$ is an integral domain. The isomorphism is defined as follows. 
For every  $\psi \in \mathcal{A}(\mathcal{U}_{pc})$ let us define the function $\widetilde{\psi}: \widetilde{\mathcal{U}_{pc}} \rightarrow \mathbb{R}$, such that $\widetilde{\psi}(v)=\psi(u)(T)$ for any piecewise-constant function $v: [0,T] \rightarrow \mathcal{U}$, where $u$ is any piecewise-constant function defined on $[0,+\infty[$ such that the restriction of $u$ to $[0,T]$ equals $v$. Then the map $\psi \mapsto \widetilde{\psi}$ is an algebraic isomorphism from $\mathcal{A}(\mathcal{U}_{pc})$ to $\mathcal{A} \big( \widetilde{\mathcal{U}_{pc}} \rightarrow \mathbb{R} \big)$.
\end{Remark}

\begin{Definition} \label{derivation_real_function}
Let $\varphi \in \mathcal{A}(\mathcal{U}_{pc})$  and  define the \emph{derivative} $D_{\alpha} \varphi$ of $\varphi$ along $\alpha \in \mathcal{U}$ as the function $D_{\alpha} \varphi: \mathcal{U}_{pc} \rightarrow PC([0;+\infty[, \mathbb{R})$, such that for all $u \in \mathcal{U}_{pc}$,
for all $t \geqslant 0$, 
\begin{equation*}
\begin{split}
& \big( D_{\alpha} \varphi(u) \big)(t) = \frac{d}{ds} \Big( \varphi \big( u_{\alpha})(t+s)\Big)_{| \, s = 0} \, , \\
& u_{\alpha}(\tau)=\left\{\begin{array}{rl} 
                     u(\tau) & \tau \in [0,t[ \\
                     \alpha  & \tau > t \\
     \end{array}\right.       
\end{split}
\end{equation*}
\end{Definition}

It is easy to see that $D_{\alpha} \varphi$ is also causal and analytic, and hence $D_{\alpha} \varphi$ belongs to $\mathcal{A}(\mathcal{U}_{pc})$. Now we define the observation algebra of an input-output map.

\begin{Definition} \label{observation_algebra_field}
Let $p: \mathcal{U}_{pc} \rightarrow PC([0; + \infty[, \mathbb{R}^{\mathrm{p}})$ be an analytic and causal input-output map. The \emph{observation algebra} of $p$, denoted by $\mathcal{A}_{obs}(p)$, is the smallest sub-algebra of the algebra $\mathcal{A} \big( \mathcal{U}_{pc} \big)$ such that the following holds.
\begin{itemize}
\item
  Consider the components $p_k: \mathcal{U}_{pc} \rightarrow PC([0; + \infty[, \mathbb{R})$, $k=1,\ldots,\mathrm{p}$ of $p$, i.e., $\forall u \in \mathcal{U}_{pc}: p(u)=(p_1(u),\ldots,p_{\mathrm{p}}(u))^T$. For 
  every $k=1,\ldots,\mathrm{p}$, $p_k \in \mathcal{A}_{obs}(p)$.
 \item
   For every $g \in A_{obs}(p)$, $D_{\alpha} g \in A_{obs}(p)$, $\alpha \in \mathcal{U}$, i.e.,
   $A_{obs}(p)$ is closed under taking derivatives $D_{\alpha}$, $\alpha \in \mathcal{U}$.
 \end{itemize}
We call the \emph{observation field}, denoted by $\mathcal{Q}_{obs}(p)$, the field of fractions of $\mathcal{A}_{obs}(p)$. 
\end{Definition}

\begin{Remark} \label{remark_algebra_iomap_response-map}
Observation algebra / field have already been introduced in \cite[Definition 5.9]{JanaSIAM} or in \cite[Definition 4.7]{Jana} for response maps. Moreover we know that the field $\mathcal{Q}_{obs}(p)$ is well-defined because, from Remark \ref{algebra_io-map_integral_domain} we know that the algebra $\mathcal{A}( \mathcal{U}_{pc})$ is an integral domain. Thus the transcendence degree of $\mathcal{A}_{obs}(p)$, denoted by $\mathrm{trdeg} \mathcal{A}_{obs}(p)$, is well-defined, see Section \ref{def} for the definition of transcendence degree. 
\end{Remark}

Now if an input-output map $p: \mathcal{U}_{pc} \rightarrow PC([0; + \infty[, \mathbb{R}^{\mathrm{p}})$ is realized by an RNN $\Sigma = (A,B,C,\mathcal{U},\sigma,x_0)$, with $\sigma$ satisfying Assumption $(A1)$, then the rational system $\mathscr{R}(\Sigma)$, given in Definition \ref{rational_system_R(Sigma)}, also realizes the input-output map $p$ by Theorem \ref{Theoremi:main} and Lemma \ref{lemma:assum:eq}. Thus the transcendence degree of the observation algebra $\mathcal{A}_{obs}(p)$ of $p$ should be necessary finite by \cite[Theorem 5.16]{JanaSIAM}. We are now in the position to state a necessary condition for existence of an RNN realization, which summarizes the arguments above.

\begin{Theorem}[Existence of an RNN: necessary condition] \label{Existence_realization_RNN}
The input-output map $p: \mathcal{U}_{pc} \rightarrow PC([0; + \infty[, \mathbb{R}^{\mathrm{p}})$ has a realization by an RNN whose activation function satisfies Assumption $(A1)$ only if $p$ is causal, analytic and $\mathrm{trdeg}~  A_{obs}(p) < +\infty$.
\end{Theorem}
The proof of Theorem \ref{Existence_realization_RNN} is presented in  \cite{TMArxive}.


\section{Minimality, reachability and observability of RNNs} 
\label{sect:min}

In this section, we first provide sufficient conditions for minimality of a given RNN $\Sigma$ by assuming minimality of the rational system $\mathscr{R}(\Sigma)$ given in Definition \ref{rational_system_R(Sigma)}. Since the latter rational system is often non-minimal, we introduce in the sequel another rational system, called the \emph{auxiliary rational system} of the RNN $\Sigma$, denoted by $\mathscr{R}_{aux}(\Sigma)$, also used for providing sufficient minimal conditions of $\Sigma$. Then we provide a Hankel-rank like condition for minimality of the RNN $\Sigma$. 
Finally we relate reachability and observability properties of $\mathscr{R}_{aux}(\Sigma)$, well-known for rational systems, to similar properties for $\Sigma$, namely \emph{span-reachability} introduced later in this paper and a necessary condition of observability provided in \cite[Theorem 1]{ASo}. 

\subsection{Sufficient conditions for minimality of RNNs}


As the first step, we define the notion of dimension for RNNs and rational systems. Let $\Sigma$ be an RNN as in Definition \ref{neural_network}. The \emph{dimension} of $\Sigma$ is the dimension of its state-space (i.e. the number of states), and it is denoted by $dim(\Sigma)$. In this case, we have $dim(\Sigma) = n$. Consider a rational system $\mathscr{R}$ of the form \eqref{diff_equation_rational_system}. The dimension of $\mathscr{R}$, denoted by $dim(\mathscr{R})$, is here defined as the number of state-variables, i.e. $dim(\mathscr{R}) = n$.
\begin{Remark}[Dimension of rational systems]
In \cite[Definition 13]{JanaVS} rational systems were defined as systems state-space of which is an irreducible algebraic
variety and the dimension of a rational system was defined as the transcendence degree of the ring of all polynomial functions on this variety. 
In our case the state-space of the system is $\mathbb{R}^n$ which is a trivial algebraic variety.  Our definition of dimension coincides with that of \cite[Definition 13]{JanaVS}, as the transcendence degree of the ring of polynomials on
$\mathbb{R}^n$ is $n$. 
\end{Remark}
Now we are able to define minimal RNN realization and minimal rational realization.
\begin{Definition}[Minimality]
We say that a rational system $\mathscr{R}$ is a \emph{minimal realization} of an input-output map $p$, if $\mathscr{R}$ is a realization of $p$ and there exists no rational system $\mathscr{R}^{'}$ such that $\mathscr{R}^{'}$ is a realization of $p$ and $\dim(\mathscr{R}^{'}) < \dim(\mathscr{R})$. An RNN $\Sigma$ with activation function $\sigma$ is said to be a \emph{$\sigma$-minimal} realization of an input-output function $p$, if $\Sigma$ is a realization of $p$ and there exists no RNN $\Sigma^{'}$ with activation function $\sigma$, such that $\Sigma^{'}$ is a realization of $p$ and $\dim(\Sigma^{'}) < \dim(\Sigma)$.
\end{Definition}
By considering the rational system $\mathscr{R}(\Sigma)$, we provide sufficient condition for minimality of RNNs as follows.
\begin{Lemma}
\label{minimality_1}
Let $\Sigma = (A, B, C, \sigma, x_0)$ be an RNN, whose activation function $\sigma$ satisfies $(A1)$ and assume that $\Sigma$ is a realization of the input-output map $p$. If the rational system $\mathscr{R}(\Sigma)$, given in Definition \ref{rational_system_R(Sigma)}, is a minimal realization of $p$, then $\Sigma$ is a minimal RNN realization of $p$.
\end{Lemma}
The proof of Lemma \ref{minimality_1} is presented in  \cite{TMArxive}.

Unfortunately, in most of the cases, $\mathscr{R}(\Sigma)$ will not be minimal. Intuitively, this has to do with the fact that the states $x_1(t), \ldots, x_n(t)$ of $\mathscr{R}(\Sigma)$ are integrals of the other states, leading to lack of observability if $n > 1$ and $x(0)$ is chosen so that $Cx(0)=0$. 

In order to remedy this problem,  we introduce another rational system, called the \emph{auxiliary rational system} depending on the RNN $\Sigma$ with less components in the state. In order to define this rational system without excessive notation, we will use the same convention as for defining $\mathscr{R}(\Sigma)$, namely we identify the sum of fraction of multi-variable polynomials $\sum_{k=1}^{N} \frac{P_i}{Q_i}$ with the fraction obtained by bringing all summands to the same denominator, i.e., $\frac{\sum_{k=1}^{N} P_i\Pi_{r=1, r \ne k}^{N}Q_i}{\Pi_{k=1}^{N} Q_i}$. 

\begin{Definition} \label{Definition_R_aux_Sigma}
Let $\Sigma=(A,B,C,\mathcal{U},\sigma,x_0)$ be an RNN, whose activation function satisfies $(A1)$, assume that
$\xi=(\xi_1,\ldots, \xi_N)$ satisfies \eqref{Assumption1}. Define the \emph{ auxiliary rational system $\mathscr{R}_{aux}(\Sigma)$ associated with the RNN $\Sigma$}, $A=(a_{i,j})_{i,j=1}^{n}, C=(c_{k,i})_{k=1,\ldots,p,i=1,\ldots,n}$  as follows:
\begin{align*}
& \forall i=1,\ldots,N, ~ j=1,\ldots,n, ~ \alpha \in U: \\
& \dot{\upsilon}_{i,j,\alpha}(t) = \frac{U_i(\upsilon_{j,\alpha}(t))}{V_i(\upsilon_{j,\alpha}(t))} \, \{ \, \sum_{l=1}^n a_{j,l} \frac{U_0(\upsilon_{l,\beta}(t))}{V_0(\upsilon_{l,\beta}(t))} \, \} \, , \; \mbox{if} \; u(t) = \beta\\
& \upsilon_{j,\alpha}(t) = (\upsilon_{1,j,\alpha}(t), \ldots, \upsilon_{n,j,\alpha}(t))^T \, ,\\
& \upsilon_{j,\alpha}(0)=\xi \big( e_j^T (A x_0 + B\alpha) \big) \, , \\
& y_{k,\alpha}(t) = \sum_{i=1}^n c_{k,i} \, \frac{U_0(\upsilon_{i,\alpha}(t))}{V_0(\upsilon_{i,\alpha}(t))} \, , ~ k=1,\ldots,\mathrm{p} \\
\end{align*}
\end{Definition}

\begin{Remark}[Polynomial $\mathscr{R}_{aux}(\Sigma)$]
If Assumption $(A1)$ is satisfied with polynomial equations, i.e. $V_i=1$, $i=0,\ldots,N$, like in examples from Example \ref{example_activation_functions}, then $\mathscr{R}_{aux}(\Sigma)$ is a polynomial system.
\end{Remark}

Note that the rational system $\mathscr{R}_{aux}(\Sigma)$, given in Definition \ref{Definition_R_aux_Sigma}, does not realize the input-output map $p_{\Sigma, x_0}$ of the RNN $\Sigma = (A,B,C,\mathcal{U},\sigma,x_0)$ in general. Roughly speaking, it realizes the input-output map constructed with the derivatives of $p_{\Sigma,x_0}$ along $\alpha_1, \ldots, \alpha_K \in \mathcal{U}$ in the sense of Definition \ref{derivation_real_function}. We now formalize it properly. Let $p: \mathcal{U}_{pc} \rightarrow PC([0;+\infty[,\mathbb{R}^{\mathrm{p}})$ be an input-output map realized by the RNN $\Sigma$. Define the input-output map $\hat{p}: \mathcal{U}_{pc} \rightarrow PC([0;+\infty[,\mathbb{R}^{\mathrm{p} K})$, where $K = | \mathcal{U} |$, as follows:
\begin{equation} \label{input-output_map_derivatives}
\forall u \in \mathcal{U}_{pc} \, , \quad \hat{p}(u) = (D_{\alpha_1}p(u), \ldots, D_{\alpha_K}p(u))^T \, ,
\end{equation}
with, for $\alpha \in \mathcal{U}$, $D_{\alpha}p(u) = (D_{\alpha}p_k(u))_{1 \leqslant k \leqslant \mathrm{p}}$, where $D_{\alpha} p_k(u)$ is defined in Definition \ref{derivation_real_function}. 

\begin{Lemma} \label{realization_of_R_aux_Sigma}
Let $\Sigma = (A,B,C,\mathcal{U},\sigma,x_0)$ be an RNN and let $p: \mathcal{U}_{pc} \rightarrow PC([0;+\infty[,\mathbb{R}^{\mathrm{p}})$ be an input-output map. If the RNN $\Sigma$ realizes $p$, then the rational system $\mathscr{R}_{aux}(\Sigma)$ realizes the input-output map $\hat{p}$ defined in \eqref{input-output_map_derivatives}.
\end{Lemma}
The proof of Lemma \ref{realization_of_R_aux_Sigma} is presented in  \cite{TMArxive}.

\begin{Lemma} \label{RNN_minimality}
Let $\Sigma = (A,B,C,\mathcal{U},\sigma,x_0)$ be an RNN and let $p: \mathcal{U}_{pc} \rightarrow PC([0;+\infty[,\mathbb{R}^p)$ be an input-output map realized by $\Sigma$. If the rational system $\mathscr{R}_{aux}(\Sigma)$ is a minimal realization of $\hat{p}$, then the RNN $\Sigma$ is $\sigma$-minimal realization of $p$.
\end{Lemma}
The proof of Lemma \ref{RNN_minimality} is presented in  \cite{TMArxive}.

It is known from \cite{Jana,NashSIAM} that the properties of \emph{algebraic, rational and semi-algebraic observability} and \emph{algebraic reachability} characterize minimality of rational systems and these properties can be checked using methods of computational algebra \cite{JanaCDC2016}. In particular, we can derive sufficient conditions for the minimality of the RNN $\Sigma$ using these reachability and observability concepts for rational systems. 

In order to explore these sufficient conditions in more details, we will recall below the notions of algebraic reachability and algebraic/semi-algebraic observability for rational systems. Define the set of reachable states of a  rational system $\mathscr{R}$ of the form \eqref{diff_equation_rational_system} as $\mathrm{R}_{\mathscr{R}}(\upsilon_0)$:
\begin{equation*} 
\begin{split}
& \mathrm{R}_{\mathscr{R}}(\upsilon_0) = \{ \upsilon(t) \; | \; t \geqslant 0, (\upsilon,u,y) \\
& \mbox{ is a solution of } \mathscr{R}, \upsilon(0)=\upsilon_0\} \, .
\end{split}
\end{equation*}
The system $\mathscr{R}$ is said to be \emph{algebraically reachable}, if there is no non-trivial polynomial which is zero on
$\mathrm{R}_{\mathscr{R}}(\upsilon_0)$. The system $\mathscr{R}$ is called accessible, if $\mathrm{R}_{\mathscr{R}}(\upsilon_0)$ contains an open subset of $\mathbb{R}^n$. It is clear that accessibility of $\mathscr{R}$ implies algebraic reachability of $\mathscr{R}$. 

For a rational system $\mathscr{R}$ as in Definition \ref{rational_system} with state-space $\mathbb{R}^n$, recall from \cite[Definition 3.19]{Jana} or from \cite[Definition 4]{Bartoszewicz1} that \emph{observation algebra} of $\mathscr{R}$, denoted by $\mathcal{A}_{obs}(\mathscr{R})$, is the smallest sub-algebra of the field of rational functions $R(X_1,\ldots,X_n)$ which contains $\frac{h_{k,1}}{h_{k,2}}$, $k=1,\ldots,\mathrm{p}$
and which is closed under taking the formal Lie derivatives with respect to the formal vector fields
$f_{\alpha}=\sum_{i=1}^{n} \frac{P_{i,\alpha}}{Q_{i,\alpha}} \frac{\partial}{\partial X_i}$. 
If $\mathscr{R}$ is polynomial, i.e. $Q_{i,\alpha}=1$, $i=1,\ldots,n$, $h_{k,2}=1$, $k=1,\ldots,\mathrm{p}$, then $\mathcal{A}_{obs}(\mathscr{R})$ is the sub-algebra of the ring of polynomials $R[X_1,\ldots,X_n]$. 
Following \cite{Jana} we say that the rational system $\mathscr{R}$ is \emph{algebraically observable}, if $\mathcal{A}_{obs}(\mathscr{R}) = R[X_1,\ldots,X_n]$. Following \cite{NashSIAM} that $\mathscr{R}$ is \emph{semi-algebraically observable} if $\mathrm{trdeg}(\mathcal{A}_{obs}(\mathscr{R}))=n$. We will say that $\mathscr{R}$ is \emph{observable}, if for every two distinct initial states $\upsilon_0,\upsilon_0^{'}$ there exists solutions $(\upsilon,u,y)$ and $(\upsilon^{'},u,y^{'})$ of $\mathscr{R}$ such that $\upsilon(0)=\upsilon_0$, $\upsilon^{'}(0)=\upsilon_0^{'}$, and $y \ne y^{'}$. It is easy to see that algebraic observability implies semi-algebraic observability. Moreover, for polynomial systems algebraic observability implies observability \cite{Bartoszewicz}. 


Recall from \cite[Theorem 4]{JanaVS} that  a rational system $\mathscr{R}$ is minimal, if and only if it is algebraically reachable and semi-algebraically observable.



\begin{Lemma}[Sufficient conditions for minimality]
\label{suff:cond:minimality}
If one of the conditions below holds, then $\Sigma$ is $\sigma$-minimal realization of $p$: 
\begin{itemize}
\item $\mathscr{R}_{aux}(\Sigma)$ is semi-algebraically observable and algebraically reachable. 
\item $\mathscr{R}_{aux}(\Sigma)$ is polynomial, it is algebraically observable and algebraically reachable. 
\item $\mathscr{R}_{aux}(\Sigma)$ is polynomial, it is algebraically observable and accessible.
\end{itemize}
\end{Lemma}
The above lemma is then a direct consequence of Lemma \ref{RNN_minimality} and \cite[Proposition 6]{JanaVS}, its detailed proof is presented in  \cite{TMArxive}.

\subsection{A Hankel-rank like condition for minimality of RNNs}

In this part, we relate minimality of an RNN $\Sigma$ with the transcendence degree of the observation algebra $\mathcal{A}_{obs}(p)$ defined in Definition \ref{observation_algebra_field}. Recall that a linear system is a minimal realization of its input-output map, if and only if the dimension of this system equals the rank of the Hankel-matrix constructed from the Markov parameters of this input-output map. We would like to formulate a similar result, where the role of the rank of the Hankel-matrix is played by the observation algebra $\mathcal{A}_{obs}(p)$. To this end, recall from \cite[Lemma 1, Theorem 4]{JanaVS} that $\mathscr{R}(\Sigma)$ is minimal if and only if $\dim(\mathscr{R}(\Sigma))= \mathrm{trdeg} \mathcal{A}_{obs}(p)$. In a similar manner, if $\mathcal{A}_{obs}(\hat{p})$ is the observation algebra of the input-output map $\hat{p}$ defined in \eqref{input-output_map_derivatives}, then $\mathscr{R}_{aux}(\Sigma)$ is minimal if and only if $\dim(\mathscr{R}_{aux}(\Sigma))= \mathrm{trdeg} \mathcal{A}_{obs}(\hat{p})$. Note that $\mathcal{A}_{obs}(\hat{p})$ is a sub-algebra of $\mathcal{A}_{obs}(p)$ generated by the elements of the form $D_{\alpha_1}\cdots D_{\alpha_l} p_k$, $l > 0$, $k=1,\ldots,\mathrm{p}$, $\alpha_1,\ldots,\alpha_l \in \mathcal{U}$. The following lemma is then a direct consequence of Lemma \ref{minimality_1} and Lemma \ref{RNN_minimality}.

\begin{Lemma}[Hankel-rank like conditions for minimality]
\label{suff:hankel}
Assume that $\sigma$ satisfies $(A1)$, and let the RNN $\Sigma=(A,B,C,\mathcal{U},\sigma,x_0)$ be a realization of the input-output map $p$. 
If one of the following conditions hold for $n=\dim(\Sigma)$:
\begin{itemize}
\item $\mathrm{trdeg} \mathcal{A}_{obs}(p)=n(1+|\mathcal{U}|N)$, or
\item $\mathrm{trdeg} \mathcal{A}_{obs}(\hat{p})=n|\mathcal{U}|N$,
\end{itemize}
then  $\Sigma$ is a $\sigma$-minimal realization of $p$. 
\end{Lemma}
The proof of Lemma \ref{suff:hankel} is presented in  \cite{TMArxive}.

\subsection{Some aspects of reachability and observability of RNNs}

One may wonder how restrictive the conditions of Lemma \ref{suff:cond:minimality} are, and how they relate to accessibility/reachability and observability of the RNN $\Sigma$ studied in \cite{ASo,QSc,ADPa}. In fact, observability and reachability properties of $\mathscr{R}_{aux}(\Sigma)$ imply similar properties of the RNN $\Sigma$. In order to present this relationship more precisely, we introduce the following terminology. Define the reachable set of an RNN $\Sigma=(A,B,C,\mathcal{U},\sigma, x_0)$
\begin{equation*}
\begin{split}
& \mathrm{R}_{\Sigma}(x_0) = \{x(t) \; | \; t \geqslant 0 \, , (x,u,y) \\
& \mbox{ is a solution of $\Sigma$}, x(0)=x_0  \}. 
\end{split}
\end{equation*}
We will say that $\Sigma$ is \emph{accessible}, if $\mathrm{R}_{\Sigma}(x_0)$ contains an open subset of $\mathbb{R}^n$, we say that $\Sigma$ is \emph{algebraically reachable} if there is no non-trivial polynomial
which is zero on $R_{\Sigma}(x_0)$. We say that $\Sigma$ is \emph{span-reachable}, if the linear span of the elements $\mathrm{R}_{\Sigma}(x_0)$ is $\mathbb{R}^n$, i.e. $\Sigma$ is reachable if there exist no linear function which is zero on $\mathrm{R}_{\Sigma}(x_0)$. Clearly, if $\Sigma$ is accessible, then it is algebraically reachable, and if $\Sigma$ is algebraically reachable, then it is span-reachable. We say that the RNN $\Sigma$ is \emph{weakly observable} if for every initial state $\hat{x} \in \mathbb{R}^n$ there is an open subset $V$ of $\mathbb{R}^n$ such that $\hat{x} \in V$ and for every $\hat{x} \ne \overline{x} \in V$, there exist solution $(x,u,y)$ and $(x',u,y')$ of $\Sigma$, with $x(0)=\hat{x}$ and $x'(0)=\overline{x}$, such that $y \neq y'$. Then the RNN $\Sigma$ is \emph{observable} if for every initial state  $\hat{x} \in \mathbb{R}^n$, $V = \mathbb{R}^n$ in the latter definition.

\begin{Lemma} \label{RNN_ReachObs}
Let $\Sigma = (A,B,C,\mathcal{U},\sigma,x_0)$ be an RNN.
\begin{itemize}
\item
If $\mathscr{R}_{aux}(\Sigma)$ is algebraically reachable, then 
$\Sigma$ is span-reachable.  In particular, if $\mathscr{R}_{aux}(\Sigma)$ is accessible, then  $\Sigma$ is span-reachable. 
\item
If $\mathscr{R}_{aux}(\Sigma)$ is polynomial, and it is observable, and if the function $\sigma$ is invertible and $Ker(A)$ is trivial, then $\Sigma$ is observable. 
In particular, if $\mathscr{R}_{aux}$ is algebraically observable, then $\Sigma$ is observable. 
\item
If $\mathscr{R}_{aux}(\Sigma)$ is polynomial, and it is semi-algebraically observable, and if the function $\sigma$ is invertible and $Ker(A)$ is trivial, then $\Sigma$ is weakly observable. 
\end{itemize}
\end{Lemma}
The proof of Lemma \ref{RNN_ReachObs} is presented in  \cite{TMArxive}.
\begin{Remark}[Invertibility of $\sigma$]
Note that assuming that the activation function $\sigma$ is invertible is not too restrictive as it holds for many commonly used activation functions, see Example \ref{example_activation_functions}.
\end{Remark}
Observe that accessibility, and algebraic / semi-algebraically observability conditions for rational / polynomial systems can be checked by using methods of computer algebra \cite{JanaCDC2016}. 
In contrast, for checking accessibility and (weak) observability of an RNN the only systematic tools are 
the rank conditions \cite[Theorem 2.2, Theorem 2.5, Theorem 3.1, Theorem 3.5]{NLCO} or \cite[Corollary 2.2.5, Corollary 2.3.5]{Isidori}, which are not computationally effective. 
\begin{Remark}[Minimality of RNN as an analytic system]
\label{rem:analytic}
From \cite[Theorem 1.12]{J} it follows that if the RNN $\Sigma$  is accessible and weakly observable, then it is a minimal realization of its input-output map $p$.  From the comparison between the conditions of  Lemma \ref{suff:cond:minimality} with those of Lemma \ref{RNN_ReachObs} it
is clear that minimality of $\mathscr{R}_{aux}(\Sigma)$ is a much weaker condition than
accessibility and weak observability of $\Sigma$. This suggests that using realization theory of rational systems is likely to yield more useful results for RNNs than using realization theory of general analytic systems. 
\end{Remark}

Recall from \cite[Theorem 1]{ASo} that a necessary condition for observability of $\Sigma = (A,B,C,\mathcal{U},\sigma,x_0)$ is that the largest $A$-invariant coordinate subspace of $\Sigma$ included in $Ker(C)$ is trivial. More precisely, following \cite{ASo} we say that a vector subspace $V$ of $\mathbb{R}^n$ is a \emph{coordinate subspace} if it is spanned by some vectors from the canonical basis of $\mathbb{R}^n$, i.e. there exists an integer $s > 0$ and integers $i_1, \ldots, i_s \in \{1, \ldots, n\}$ such that $V$ is spanned by $e_{i_1}, \ldots, e_{i_s}$, where $(e_1, \ldots, e_n)$ denotes the canonical basis of $\mathbb{R}^n$. We write $\mathcal{O}_c(A,C)$ the largest coordinate subspace which is $A$-invariant and contained in $Ker(C)$.

\begin{Lemma} \label{semi-alg-obs_R_aux}
If $\mathscr{R}_{aux}(\Sigma)$ is polynomial and it is semi-algebraically observable, then there exists no non-trivial coordinate subspace which is $A$-invariant and contained in $Ker(C)$, i.e. $\mathcal{O}_c(A,C)=\{0\}$.
\end{Lemma}
The proof of Lemma \ref{semi-alg-obs_R_aux} is presented in  \cite{TMArxive}.

By \cite[Theorem 1]{ASo}  $\mathcal{O}_c(A,C)=\{0\}$, it is also sufficient if $\ker(C) \cap \ker(A)=\{0\}$, the activation function $\sigma$ satisfies only the IPP property, given in \cite{ASo} for example, and if $B$ verifies a condition on its rows. But here we do not need the latter hypothesis on $\Sigma$.


\begin{Example} \label{example_RNN_semi-alg-obs}
Let $\sigma: \mathbb{R} \rightarrow \mathbb{R}$ be the sigmoid function as in Example \ref{example_activation_functions}. Consider the RNN with activation function $\sigma$, defined as follows:
\begin{equation*}
\Sigma \; : \; \left\lbrace \begin{split}
\dot{x_1} = \sigma(x_2 + u), \quad 
\dot{x_2} = \sigma(x_1 + u) \\
x_1(0) = x_2(0) = 0, \quad y = x_1
\end{split}\right.
\end{equation*}
Here $A = \left( \begin{array}{cc}
0 & 1\\
1 & 0
\end{array} \right)$, $B = (1, 1)^T$, $C = (1, 0)$ and $u \in \mathbb{R} = \mathcal{U}$ is a fixed real number. The auxiliary rational system associated with $\Sigma$ is then given by
\begin{equation*}
\mathscr{R}_{aux}(\Sigma) \; : \; \left\lbrace \begin{split}
&\dot{\upsilon_1} = \upsilon_1 \upsilon_2 (1 - \upsilon_1), \quad \dot{\upsilon_2} = \upsilon_1 \upsilon_2 (1 - \upsilon_2)\\
& \upsilon_1(0) = \upsilon_2(0) = \sigma(\frac{1}{2} + u)\\
& \hat{y} = \upsilon_1
\end{split}\right.
\end{equation*}
where $\upsilon_k = \dot{x_k}$ for $k=1,2$. Denote by $f$ the vector field generated by $\mathscr{R}_{aux}(\Sigma)$. The output map is here $h(\upsilon_1, \upsilon_2) = \upsilon_1$, simply written $h = \upsilon_1 \in \mathcal{A}(\mathscr{R}_{aux}(\Sigma))$. We clearly have $L_f h = h \, \upsilon_2 ( 1 - h) \in \mathcal{A}(\mathscr{R}_{aux}(\Sigma))$, where $L_f$ is the Lie derivative operator along the vector field $f$. Moreover we get $\upsilon_2 = \frac{L_f h}{h(1-h)}$, and the latter belongs to the field of fractions of $\mathcal{A}(\mathscr{R}_{aux}(\Sigma))$. Thus the latter field is equal to $\mathbb{R}(\upsilon_1, \upsilon_2)$, which shows that $\mathscr{R}_{aux}(\Sigma)$ is semi-algebraically observable. By Lemma \ref{RNN_ReachObs}, the RNN $\Sigma$ is weakly observable.
\end{Example}

\section{Conclusions}

We have shown that input-output maps of a large class of recurrent neural networks can be represented by rational/polynomial systems, and we used this fact to derive necessary and sufficient conditions for existence of a realization by a recurrent neural network and its minimality. Future research will be directed towards deriving a more complete realization theory of recurrent neural network and for using the results of realization theory for analyzing machine learning algorithms.

\appendix

\section{Proofs}

In this appendix, we write all the technical proofs.

\begin{proof}[Proof of Lemma  \ref{lemma:assum:eq}]
\textbf{If $\sigma$ verifies $(A1)$ , it satisfies $(A2)$} \\
  Consider the algebra $A_{obs}(\xi_1,\ldots,\xi_N)$ genereted by $\xi_1,\ldots,\xi_N$ and their
  high-order derivatives, i.e., $A_{obs}(\xi_1,\ldots,\xi_N)$ is generated by $\{\xi_j^{(k)}\}_{k \in \mathbb{N},j=1}^{N}$. Since $\xi_1,\ldots,\xi_N$ are analytic, $A_{obs}(\xi_1,\ldots,\xi_N)$ is the sub-algebra of
 the algebra of all real analytic functions, and hence it is an integral domain.
  Note that $\dot \xi_i,\sigma$ is algebraic over $\xi_1,\ldots,\xi_n$.
Indeed, it suffices to take $Q_k(X_1, \ldots, X_{N+1}) = X_{N+1} V_k(X_1, \ldots X_N) - U_k(X_1, \ldots, X_N)$, for $0 \ge k \leqslant N$, and then $Q_i(\dot \xi_i,\xi_1,\ldots,\xi_N)=0$, $i=1,\ldots,N$ and
$Q_O(\sigma,\xi_1,\ldots,\xi_N)=0$
By taking the $r$th derivative of $Q_i(\dot \xi_i,\xi_1,\ldots,\xi_N)$, $Q_O(\sigma,\xi_1,\ldots,\xi_N)=0$,
we can conclude that the $r$th derivative $\xi^{(r)}_i$, $\sigma^{(r)}$ of $\xi_i$ and $\sigma$ respectively, are algebraic  over $\{\xi_k^{(l)}\}_{l=1,k=1}^{r-1,N}$ and $\{\xi_k^{l},\sigma^{(l)}\}_{l=1,k=1}^{r-1,N}$
, and hence, by induction on $r$, we can conclude that $\xi_i^{(r)}$, $\sigma^{(r)}$ 
are algebraic over $\xi_1,\ldots,\xi_N$. Hence, 
the algebra $A_{obs}(\sigma)$ generated by $\{\sigma^{(r)}\}_{r=0}^{\infty}$ is algebraic
over $\xi_1,\ldots,\xi_N$ and hence the transcedence degree of $A_{obs}(\Sigma)$ is at most $N$.
The latter means that there exist a non-zero polynomial $Q$ such that 
$Q(\sigma,\ldots,\sigma^{(N)})=0$, i.e. Assumption $(A2)$ holds.

\textbf{If $\sigma$ satisfies $(A2)$, then it satisfes $(A2)$}.
 Let $k \in \{0, \ldots, N\}$. It suffices here to take derivatives the equation \eqref{Assumption2} in Assumption $(A2)$, as follows:
\begin{equation*}
\begin{split}
& \frac{\partial Q}{\partial X_{N+1}}(\sigma_1, \ldots, \sigma^{(N)}) \, \sigma^{(N+1)} \, + \\
& \, \sum_{l=1}^{N} \frac{\partial Q_k}{\partial X_l}(\sigma, \ldots, \sigma^{(N)})\sigma^{(l)}= 0 \, ,
\end{split}
\end{equation*}
Then set $\xi_i=\sigma^{(i-1)}$, $i=1,\ldots,N+1$ and 
$V_k=1$, $U_k=X_{k+1}$, $k=0,1,\ldots,N$, and
$V_{N+1}=\frac{\partial Q}{\partial X_{N+1}}$, and 
$U_{N+1}=\sum_{i=1}^{N} \frac{\partial Q}{\partial X_i}(X_1,\ldots,X_{N})X_{i}$.
\end{proof}


\bigskip

\begin{proof}[Proof of Lemma \ref{Theoremi:main:lemma}]
Let $\Sigma = (A, B, C, \sigma, x_0)$ be an RNN and let $(x, u, y)$ be a solution of $\Sigma$, for a given $u \in \mathcal{U}_{pc}$. Without loss of generality, we suppose that $u(t) = \beta \in \mathcal{U}$. As in the statement of Lemma \ref{Theoremi:main:lemma}, write $\upsilon(t) = F(x(t))$, for $t \geqslant 0$. Clearly we have $\upsilon_{i,j,\alpha} = \xi_i \big( e_j^T (Ax + B \alpha) \big)$, for $i \in \{1, \ldots, N\}$, $j \in \{1, \ldots, n\}$ and $\alpha \in \mathcal{U}$, where $\xi = (\xi_1, \ldots, \xi_N)$ are analytic functions as in \eqref{Assumption1}. Thus it suffices to prove that $\upsilon_{i,j,\alpha}$ satisfies the differential equation given in Definition \ref{rational_system_R(Sigma)}. By taking the first derivative, we obtain
\begin{align*}
& \dot{\upsilon}_{i,j,\alpha}(t)\\
& = \xi_i^{(1)} \big( e_j^T (A x + B \alpha) \big) \{ \sum_{l=1}^n a_{j,l} \dot{x}_l(t;\beta) \}\\
& = \frac{U_i(\upsilon_{1,j,\alpha}(t), \ldots, \upsilon_{N,j,\alpha}(t))}{V_i(\upsilon_{1,j,\alpha}(t), \ldots, \upsilon_{N,j,\alpha}(t))} \{ \sum_{l=1}^n a_{j,l} \sigma \big( e_l^T (A x + B \beta) \big) \}\\
& = \frac{U_i(\upsilon_{j,\alpha}(t))}{V_i(\upsilon_{j,\alpha}(t))} \{ \sum_{l=1}^n a_{j,l} \frac{U_0(\upsilon_{j,\beta}(t))}{V_0(\upsilon_{j,\beta}(t))} \big) \} \, ,
\end{align*}
where $\xi_i^{(1)}$ denotes the first derivative of $\xi_i$. Thus it follows that $(\upsilon, u, y)$ of a solution of the rational system $\mathscr{R}(\Sigma)$ given in Definition \ref{rational_system_R(Sigma)}.
\end{proof}

\bigskip 

\begin{proof}[Proof of Lemma \ref{realization_of_R_aux_Sigma}]
Let $p: \mathcal{U}_{pc} \rightarrow PC([0; +\infty[, \mathbb{R}^p)$ be an input-output map. Suppose that the RNN $\Sigma = (A, B, C, \sigma, x_0)$ realizes $p$, i.e. for all $u \in \mathcal{U}_{pc}$, there exists a solution $(x,u,y)$ of $\Sigma$ such that, for all $t \geqslant 0$, $y(t) = p(u)(t)$. It suffices to prove that, for all $\alpha \in \mathcal{U}$ and $k \in \{1, \ldots, p\}$, $y_{k,\alpha} = D_{\alpha} p_k(u)$ where $y_{k,\alpha}$ is defined in Definition \ref{Definition_R_aux_Sigma}, the map $p_k: \mathcal{U}_{pc} \rightarrow PC([0; +\infty[, \mathbb{R})$ is the $k$-th component of $p$, and $D_{\alpha}p_k$ is the derivative of $p_k$ along $\alpha \in \mathcal{U}$ in the sense of Definition \ref{derivation_real_function}.\\
For all $\alpha \in \mathcal{U}$, $k \in \{1, \ldots, p\}$, $u \in \mathcal{U}_{pc}$ and $t \geqslant 0$, we have
\begin{align*}
D_{\alpha}p_k(u)(t)
& = \frac{d}{ds} \Big( p_k(u_{\alpha})(t+s) \Big)_{| \; s=0}\\
& = \sum_{i=1}^n c_{k,i} \frac{d}{ds} \Big( x_i(t+s;u_{\alpha}) \Big)_{| \; s=0}\\
& = \sum_{i=1}^n c_{k,i} \, \sigma \Big( \sum_{j=1}^n a_{i,j} x_j(t;u) + e_i^T B \alpha \Big)\\
& = \sum_{i=1}^n c_{k,i} \frac{U_0(\upsilon_{1,i,\alpha}(t), \ldots, \upsilon_{N,i,\alpha}(t))}{V_0(\upsilon_{1,i,\alpha}(t), \ldots, \upsilon_{N,i,\alpha}(t))}\\
& = \sum_{i=1}^n c_{k,i} \frac{U_0(\upsilon_{i,\alpha}(t)}{V_0(\upsilon_{i,\alpha}(t))}\\
& = y_{k,\alpha}(t) \, ,
\end{align*}
as desired. Here we recall that the input $u_{\alpha} \in \mathcal{U}_{pc}$ is given in Definition \ref{derivation_real_function}.
\end{proof}

\bigskip

\begin{proof}[Proof of Lemma \ref{RNN_minimality}]
Recall that $| \mathcal{U} | = K$. We know that $dim(\mathscr{R}_{aux}(\Sigma) = n K N$. Now assume that there is a recurrent neural networks $\hat{\Sigma}$, with $\hat{n} = dim(\hat{\Sigma}) < dim(\Sigma) = n$. It is clear that we have $dim(\mathscr{R}_{aux}(\hat{\Sigma})) < dim(\mathscr{R}_{aux}(\Sigma))$, where $\mathscr{R}_{aux}(\hat{\Sigma})$ is the rational system given in Definition \ref{Definition_R_aux_Sigma}, associated with $\hat{\Sigma}$. Hence it contradicts the fact that $\mathscr{R}_{aux}(\Sigma)$ is a minimal rational realization as claimed in the statement of Lemma \ref{RNN_minimality}. Thus $\Sigma$ is a minimal RNN realization of its input-output map.
\end{proof}

\bigskip

\begin{proof}[Proof of Lemma \ref{RNN_ReachObs}]
Let $\Sigma = (A,B,C,\sigma,x_0)$ be an RNN.
\begin{enumerate}
\item Assume that $\mathscr{R}_{aux}(\Sigma)$ is algebraically reachable, i.e. there is no non-trivial polynomial vanishing in the reachable set 
\begin{equation*} 
\begin{split}
& \mathrm{R}_{\mathscr{R}_{aux}(\Sigma)}(\upsilon_0) = \{ \upsilon(t) \; | \; t \geqslant 0, (\upsilon,u,y) \\
& \mbox{ is a solution of } \mathscr{R}_{aux}(\Sigma), \upsilon(0)=\upsilon_0\} \, .
\end{split}
\end{equation*}
Hence the components of $\upsilon(t)$ are algebraically independent, for $t \geqslant 0$. Now take $u \in \mathcal{U}_{pc}$ and $(x,u,y)$ a solution of the RNN $\Sigma$. Suppose that, for $t \geqslant 0$, $x_1(t), \ldots x_n(t)$ are linearly dependant. Without loss of generality, we can say that there are real values not all trivial $\lambda_1, \ldots, \lambda_{n-1}$ such that
\begin{equation*}
x_n(t) = \sum_{i=1}^{n-1} \lambda_i x_i(t) \, .
\end{equation*}
By taking the first derivative of the latter equation, we get
\begin{equation*}
\displaystyle \dot{x_n}(t) = \sum_{i=1}^{n-1} \lambda_i \dot{x_i}(t) \, ,
\end{equation*}
which implies that we have
\begin{align*}
\displaystyle & \dot{x_n}(t) \, \prod_{l=1}^n V_0(v_{l, u(t)}) = \sum_{i=1}^{n-1} \lambda_i \dot{x_i}(t) \, \prod_{l=1}^n V_0(v_{l, u(t)})\\[2mm]
& \Leftrightarrow U_0(v_{n,u(t)}) \, \prod_{l=1}^{n-1} V_0(v_{l, u(t)})\\ 
& \hspace*{1cm} - \sum_{i=1}^{n-1} U_0(v_{i,u(t)}) \, \prod_{l=1, l \neq i}^n V_0(v_{l, u(t)}) = 0 \, .
\end{align*}
This is a contradiction by hypothesis of the statement. Finally $x_1(t), \ldots, x_n(t)$ are linearly independent. Thus it says that the RNN is span-reachable.

\item Assume that $\mathscr{R}_{aux}(\Sigma)$ is polynomial and observable in the sense of distinguishable states. Moreover suppose that the activation function $\sigma$ is invertible and that $Ker(A)$ is trivial. Take two initial states $\overline{x}, \hat{x} \in \mathbb{R}^n$ such that there exist solutions $(x,u,y)$ and $(x',u,y')$ of $\Sigma$, such that $x(0)=\overline{x}, x'(0)=\hat{x}$ and $y = y'$. We now prove that $\overline{x} = \hat{x}$. By using Lemma \ref{Theoremi:main:lemma} and Definition \ref{Definition_R_aux_Sigma}, it is easy to prove that there are solutions $(\upsilon,u,(y_{k,\alpha})_{k,\alpha})$ and $(\upsilon',u,(y_{k,\alpha}')_{k,\alpha})$ of the rational system $\mathscr{R}_{aux}(\Sigma)$, with $\upsilon_{j,\alpha}(0) = \xi \big( e_j^T (A \overline{x} + B \alpha) \big)$ and $\upsilon_{j,\alpha}'(0) = \xi \big( e_j^T (A \hat{x} + B \alpha) \big)$, such that $y_{k,\alpha} = y_{k,\alpha}'$, for $j \in \{1, \ldots, n\}$, $k \in \{1,\ldots, p\}$ and $\alpha \in \mathcal{U}$. As $\mathscr{R}_{aux}(\Sigma)$ is polynomial, it follows that we have
\begin{align*}
\sigma \big( e_j^T (A \overline{x} + B \alpha) \big)
& = U_0(\upsilon_{j,\alpha}(0))\\
& = U_0(\upsilon_{j,\alpha}'(0))\\
& = \sigma \big( e_j^T (A \hat{x} + B \alpha) \big) \, ,
\end{align*}
by using Assumption $(A1)$. As $\sigma$ is invertible, we get $\overline{x} - \hat{x} \in Ker(A)$, which implies that $\overline{x} = \hat{x}$ because $Ker(A)=\{0\}$. Thus the RNN $\Sigma$ is observable.

\item Assume that $\mathscr{R}_{aux}(\Sigma)$ is polynomial and it is semi-algebraically observable. Moreover suppose that the activation function $\sigma$ is invertible and that $Ker(A)$ is trivial. From \cite[Proposition 4.20, Corollary 4.22]{NashSIAM}, $\mathscr{R}_{aux}(\Sigma)$ is weakly observable, i.e. for every initial state $\hat{\upsilon} \in \mathbb{R}^{n | \mathcal{U} | N}$ there exists an open set $W$ of $\mathbb{R}^{n |\mathcal{U}| N}$ such that, for all $\overline{\upsilon} \in W$, there are solutions $(\upsilon, u, y)$ and $(\upsilon',u,y')$ of $\mathscr{R}_{aux}(\Sigma)$ verifying $\upsilon(0)=\hat{\upsilon}, \upsilon'(0)=\overline{\upsilon}$ and $y \neq y'$. Consider now the map $\hat{F}: \mathbb{R}^n \rightarrow \mathbb{R}^{n | \mathcal{U} | N}, \, x \mapsto (z_1, \ldots, z_{n N | \mathcal{U} |})$, which is the composition of a projection map with the map $F: \mathbb{R}^n \rightarrow \mathbb{R}^{n | \mathcal{U} | (N+1)}$ defined in Lemma \ref{Theoremi:main:lemma}. Thus it is a continuous map, and it is straightforward to check that, for $u \in \mathcal{U}_{pc}$ and $(x,u,y)$ a solution of the RNN $\Sigma$, $(\upsilon,u,(y_{k,\alpha}))$ is a solution of the rational system $\mathscr{R}_{aux}(\Sigma)$, where $\upsilon(t) = \hat{F}(x(t))$ for $t \geqslant 0$. By continuity of the map $\hat{F}: \mathbb{R} \rightarrow \mathbb{R}^{n (| \mathcal{U} | N+1)}$, $V = (\hat{F})^{-1}(W)$ is an open set of $\mathbb{R}^n$. Thus, by using similar arguments as in the proof of second point of Lemma \ref{RNN_ReachObs}, for every initial states $\hat{x}, \overline{x} \in V$, there are solutions $(x,u,y)$ and $(x',u,y')$ of the RNN $\Sigma$ such that $x(0)=\hat{x}, x'(0)=\overline{x}$ and $y \neq y'$. Hence the RNN $\Sigma$ is weakly observable.
\end{enumerate}
\end{proof}

\bigskip

\begin{proof}[Proof of Lemma \ref{semi-alg-obs_R_aux}]
Assume that the rational system $\mathscr{R}_{aux}(\Sigma)$ is semi-algebraically observable, i.e. $\mathrm{trdeg} A_{obs}(\mathscr{R}_{aux}(\Sigma)) = n K N$.  Suppose that there is a non-trivial coordinate space $V$ which is $A$-invariant and included in $Ker(C)$. We have 
\begin{equation*}
V = span\{ e_i \; | \; i \in I\} \, , \quad \text{with} \hspace*{3mm} \emptyset \neq I \subset \{1, \ldots, n\} \, .
\end{equation*}
As $V$ is included in $Ker(C)$, it means that, for $i \in I$, the $i$-th column of $C$ is trivial. Moreover saying that $V$ is $A$-invariant means that, for $j \not\in I$ and $k \in I$, $a_{jk} = 0$. We recall that the output of the rational system $\mathscr{R}(\Sigma)$ is given as follows:
\begin{align*}
\displaystyle \forall k \in \{1, \ldots, p\} \, , \hspace*{3mm} y_k(t) 
= \sum_{i=1}^n c_{k,i} x_i(t)
= \sum_{j \in J} c_{k,j} x_j(t) \, ,
\end{align*}
where $J = \{1, \ldots, n\} \backslash I$. The output map $y_{\alpha, k}(t)$ of the rational system $\mathscr{R}_{aux}(\Sigma)$ at time $t$ is constructed from the output map of $\mathscr{R}(\Sigma)$ by taking the first derivative, by Leamma \ref{realization_of_R_aux_Sigma}. Thus, for $k \in \{1, \ldots, p\}$ and $\alpha \in \mathcal{U}$, we get
\begin{align*}
y_{k,\alpha}(t) 
= \sum_{i=1}^n c_{k,i} \, \frac{U_0(\upsilon_{i,\alpha}(t))}{V_0(\upsilon_{i,\alpha}(t))}
= \sum_{j \in J} c_{k,j} \, \frac{U_0(\upsilon_{j,\alpha}(t))}{V_0(\upsilon_{j,\alpha}(t))}
\end{align*}
for $k = 1, \ldots, p$. As $a_{j,l} = 0$ for $j \in J$ and $l \in I$, we just need to take the variables $\upsilon_{i,j,\alpha}$ with $i \in \{1, \ldots, N\}$, $j \in J$ and $\alpha \in \mathcal{U}$. Because $\upsilon_{i,j,\alpha}$ satisfies now the following differential equation:
\begin{align*}
\dot{\upsilon}_{i,j,\alpha}(t) 
& = \frac{U_i(\upsilon_{j,\alpha}(t))}{V_i(\upsilon_{j,\alpha}(t))} \, \{ \, \sum_{l=1}^n a_{j,l} \frac{U_0(\upsilon_{l,\beta}(t))}{V_0(\upsilon_{l,\beta}(t))} \, \}\\
& = \frac{U_i(\upsilon_{j,\alpha}(t))}{V_i(\upsilon_{j,\alpha}(t))} \, \{ \, \sum_{l \in J}^n a_{j,l} \frac{U_0(\upsilon_{l,\beta}(t))}{V_0(\upsilon_{l,\beta}(t))} \, \} \, , \; \mbox{if} \; u(t) = \beta \, .
\end{align*}
By some calculations, it is possible to prove that the fraction of field of $\mathcal{A}_{obs}(\mathscr{R}_{aux}(\Sigma))$ is included in the fraction of field of the ring
\begin{equation*}
\mathbb{R}[\upsilon_{i,j,\alpha} \; | \; \alpha \in \mathcal{U} \, , i \in \{1, \ldots, N\} \, , j \in J\} \, ,
\end{equation*}
so that $\mathrm{trdeg} \mathcal{A}_{obs}(\mathscr{R}_{aux}(\Sigma)) \leqslant | J | \, d \, N < n \, d N$, because $I \neq \emptyset$ implying that $| J | < n$. This is a contradiction. Hence there is no non-trivial coordinate subspace $A$-invariant included in $Ker(C)$.
\end{proof}

\end{document}